\documentclass[11pt,oneside]{amsart}
\usepackage{amsmath,amssymb,amsthm,bm}
\usepackage[alphabetic,abbrev]{amsrefs} % use AMS ref scheme
\usepackage{tikz}

\newcommand{\R}{\mathbb{R}} % real numbers
\newcommand{\C}{\mathbb{C}} % complex numbers
\newcommand{\Z}{\mathbb{Z}} % integers
\newcommand{\Ptn}{\mathcal{P}} % partition algebra
\newcommand{\End}{\operatorname{End}} % endomorphisms
\newcommand{\GL}{\operatorname{GL}} % general linear group
\newcommand{\OO}{\operatorname{O}} % orthogonal group
\newcommand{\SO}{\operatorname{SO}} % special orthogonal group
\newcommand{\M}{\operatorname{M}} % matrices
\newcommand{\ov}{\overline} % overline short form
\newcommand{\transpose}{\mathsf{T}} % transpose
\newcommand{\bE}{\mathbf{E}} % bold E
\newcommand{\bL}{\mathbf{L}} % bold L
\newcommand{\bF}{\mathbf{F}} % bold F

\newcommand{\ee}{\mathbf{e}} % bold e
\newcommand{\ff}{\mathbf{f}} % bold f
\newcommand{\bl}{\bm{\ell}} % bold l
\newcommand{\uu}{\mathbf{u}} % bold u
\newcommand{\vv}{\mathbf{v}} % bold v

\newcommand{\B}{\mathfrak{B}} % full Brauer algebra
\newcommand{\PB}{\mathcal{P}\mathfrak{B}} % partial Brauer algebra
\newcommand{\so}{\mathfrak{so}} % Lie algebra
\newcommand{\Lie}{\operatorname{Lie}} % Lie algebra of
\newcommand{\Sym}{\mathfrak{S}} % symmetric gp
\newcommand{\bil}[2]{\langle #1, #2 \rangle} % bilinear form
\newcommand{\proj}[2]{\frac{\bil{#1}{#2}}{\bil{#1}{#1}}\, #1} % projection

\swapnumbers % theorem numbers first

\newtheorem{thm}{Theorem}[section]
\newtheorem*{thm*}{Theorem}
\newtheorem{lem}[thm]{Lemma}
\newtheorem*{lem*}{Lemma}
\newtheorem{prop}[thm]{Proposition}
\newtheorem*{prop*}{Proposition}
\newtheorem{cor}[thm]{Corollary}
\newtheorem*{cor*}{Corollary}

\newtheorem*{conj*}{Conjecture}

\theoremstyle{definition}

\newtheorem*{defn*}{Definition}

\newtheorem*{example*}{Example}
\newtheorem{rmk}[thm]{Remark}
\newtheorem*{rmk*}{Remark}

\renewcommand{\labelenumi}{(\alph{enumi})}
\allowdisplaybreaks

\title[Schur--Weyl duality for twin groups]%
{Schur--Weyl duality for twin groups}
\author{Stephen Doty}
\email{doty@math.luc.edu, tonyg@math.luc.edu}
\author{Anthony Giaquinto}
\address{Department of Mathematics and Statistics,
  Loyola University Chicago, Chicago, IL 60660 USA}

\begin{document}
\begin{abstract}
The twin group $TW_n$ on $n$ strands is the group generated by $t_1,
\dots, t_{n-1}$ with defining relations $t_i^2=1$, $t_it_j = t_jt_i$
if $|i-j|>1$.  We find a new instance of semisimple Schur--Weyl
duality for tensor powers of a natural $n$-dimensional reflection
representation of $TW_n$, depending on a parameter $q$. At $q=1$ the
representation coincides with the natural permutation representation
of the symmetric group, so the new Schur--Weyl duality may be regarded
as a $q$-analogue of the one motivating the definition of the partition
algebra.
\end{abstract}
\maketitle

\section{Introduction}\noindent
Let $\bE = \C^n$ with standard basis $\{\ee_1, \dots, \ee_n\}$. The
symmetric group on $n$ letters, realized as the Weyl group $W_n$ of
permutation matrices in $\GL(\bE)$, acts as permutations on the
basis. The transposition $(i,j)$ acts as a reflection, sending
$\ee_i-\ee_j$ to its negative and fixing pointwise the orthogonal
complement (with respect to the standard bilinear form $\bil{e_i}{e_j}
= \delta_{ij})$. The line spanned by $\ee_1+\cdots+\ee_n$ is fixed
pointwise by $W_n$ and its orthogonal complement
$(\ee_1+\cdots+\ee_n)^\perp$ is an irreducible $(n-1)$-dimensional
reflection representation of $W_n$. Study of the centralizer algebra
$\End_{W_n}(\bE^{\otimes r})$ leads to the partition algebra of
\cites{Martin:book,Martin:94,Jones:94,Martin:96} and the corresponding
Deligne category \cite{Deligne}.

We are interested in $q$-analogues of the above picture. One such,
previously studied in \cite{DG}, is obtained by replacing the above
representation $W_n \to \GL(\bE)$ with the Burau representation $B_n
\to \GL(\bE)$ of Artin's braid group $B_n$. This is done by replacing
the standard bilinear form with a $q$-analogue. In effect, we perturb
the eigenvalues of the generating reflections $s_i = (i,i+1)$ from
$(1,-1)$ to parameters $(q_1,q_2)$; thus the representation $\C[B_n]
\to \GL(\bE)$ factors through the quotient map $\C[B_n] \to
H_n(q_1,q_2)$, where $H_n(q_1,q_2)$ is the two-parameter
Iwahori--Hecke algebra of \cites{Birman-Wenzl,Bigelow}.

In this paper, we study a second $q$-analogue of the situation of
paragraph one, related to that of the preceding paragraph through the
algebra $H_n(q_1,q_2)$.  Keeping $\bE$ the same, we introduce certain
operators $S_i$, preserving the bilinear form, which act on $\bE$ as
reflections fixing $(q \ee_i - \ee_{i+1})^\perp$ pointwise. This gives
a reflection representation $\rho: TW_n \to \OO(\bE)$ of the twin
group $TW_n$ defined in the Abstract. The twin group $TW_n$ and the
braid group $B_n$ are both covering groups of $W_n$, respectively
obtained by omitting the quadratic relation and cubic braid relation
from the standard Coxeter presentation of $W_n$. As long as $[n]_q =
1+q + \cdots + q^{n-1} \ne 0$, where $q = -q_2/q_1$ is the negative
ratio of the eigenvalues, $\bE = \bL \oplus \bF$ decomposes as the
direct sum of the line $\bL$ spanned by $\ee_1+\cdots + \ee_n$ and its
orthogonal complement $\bF = (\ee_1+\cdots + \ee_n)^\perp$, and these
subspaces are irreducible for both $B_n$ and $TW_n$. In case $q=1$ we
recover the situation of the first paragraph.

The twin group $TW_n$ has previously appeared in a variety of
contexts. It serves as an analogue of Artin's braid group $B_n$ in the
study of doodles \cite{Khovanov}, which are configurations of a finite
number of closed curves on a surface without triple intersections, and
it appeared in \cites{Tony:thesis,Tony:JPAA,GGS} in relation to
certain constructions of quantum groups.

When $\bE = \bL \oplus \bF$, the $r$th tensor power $\bE^{\otimes r}$
is a semisimple $\C[TW_n]$-module, with $TW_n$ acting diagonally. Our
main result is a combinatorial description of its centralizer
$\End_{TW_n}(\bE^{\otimes r})$ in case $q$ avoids a certain
well-defined set of algebraic numbers in the union of the positive
real axis and the unit circle. The centralizer is isomorphic to a
homomorphic image of the two-parameter partial Brauer algebra
$\PB_r(n,\delta')$, where $\delta' \ne 0$, studied in \cites{MM,HdM},
and is isomorphic to that algebra if $n>r$. The parameter $\delta'$
can be any nonzero scalar (all the corresponding partial Brauer
algebras are isomorphic).  This leads to the new Schur--Weyl duality
statement of Theorem \ref{thm:SWD}, extending the Schur--Weyl duality
of \cites{MM,HdM}.

The main technical fact underlying our results is Theorem
\ref{thm:density}, that (under certain restrictions on $q$) the image
$\rho(TW_n)$ of the representation is Zariski-dense in $\OO(\bL)
\oplus \OO(\bF)$. This enables an easy proof of Theorem \ref{thm:SWD}
and also gives another new instance (Theorem \ref{thm:SWD-F}) of
Schur--Weyl duality for the commuting actions of the group $TW_n$ and
the Brauer algebra $\B_r(n-1)$ on $\bF^{\otimes r}$; this is a new
variant of Brauer's original result in \cite{Brauer}.

Sections \ref{sec:prelim}--\ref{sec:twin} derive the main properties
of the twin group and its reflection representation $\bE$; an appendix
also constructs an explicit orthonormal basis and applies it to obtain
an alternative proof of the conclusion of Theorem \ref{thm:density},
under stronger hypotheses. Section \ref{sec:density} is devoted to
proving the density result of Theorem \ref{thm:density} under the
weaker hypotheses needed for the main results. Section \ref{sec:PB}
defines the partial Brauer algebra and gives its presentation by
generators and relations (due to \cite{MM}) and the main results are
deduced in Section \ref{sec:SWD}.

\section{Preliminaries}\label{sec:prelim}\noindent
The ground field is always $\C$ in this paper, unless stated
otherwise.  Begin with the two-parameter \cites{Birman-Wenzl,Bigelow}
Iwahori--Hecke algebra $H_n(q_1,q_2)$, defined by generators $T_1,
\dots, T_{n-1}$ subject to the relations
\begin{gather}
  T_iT_{i+1}T_i = T_{i+1}T_iT_{i+1}, \quad T_iT_j=T_jT_i \text{ if } |i-j|>1 \\
  (T_i-q_1)(T_i-q_2)=0.
\end{gather}
Although it is easy to eliminate one of the parameters, carrying both
causes no trouble. We assume that $q_1q_2 \ne 0$, so that the
generators $T_i$ are invertible elements in $H_n(q_1,q_2)$, with
\[
{T_i}^{-1} = (T_i-q_1-q_2)/(q_1q_2). 
\]
Artin's braid group $B_n$ may be defined by generators $\sigma_1, \dots,
\sigma_{n-1}$ subject to defining relations
\begin{equation}
\sigma_i\sigma_{i+1}\sigma_i = \sigma_{i+1}\sigma_i\sigma_{i+1}, \quad
\sigma_i\sigma_j=\sigma_j\sigma_i \text{ if } |i-j|>1.
\end{equation}
The algebra $H_n(q_1,q_2)$ is isomorphic to the quotient algebra of
the group algebra $\C[B_n]$ via the quotient map determined by
$\sigma_i \mapsto T_i$, with kernel the ideal generated by the
$(\sigma_i-q_1)(\sigma_i-q_2)$ for $i=1,\dots,n-1$.

Let $\bE$ be an $n$-dimensional complex vector space with basis
$\ee_1, \dots, \ee_n$. Consider the action of $H_n(q_1,q_2)$ defined
on generators by
\begin{equation}\label{eq:Burau}
  \begin{gathered}
   T_i \cdot \ee_{i} = (q_1+q_2) \ee_i + q_1 \ee_{i+1}, \quad
   T_i \cdot \ee_{i+1} = -q_2 \ee_i \\
   T_i \cdot \ee_j = q_1 \ee_j \; \text{ if }\; j \ne i, i+1.
  \end{gathered}
\end{equation}
In other words, if we set $Q = \begin{bmatrix} q_1+q_2 & -q_2 \\ q_1 &
  0\end{bmatrix}$ then $T_i$ acts on $\bE$ via the $n \times n$ block
  diagonal matrix
\[ \ov{T}_i = 
\begin{bmatrix}
  q_1 I_{i-1} & 0 & 0 \\
  0 & Q & 0 \\
  0 & 0 & q_1 I_{n-i-1}
\end{bmatrix} .
\]
The $\ov{T}_i$ satisfy the defining relations of $H_n(q_1,q_2)$, so
the map defined on generators by $T_i \mapsto \ov{T}_i$ is a
representation $H_n(q_1,q_2) \to \End(\bE)$.  By composing with the
quotient map $\C[B_n] \to H_n(q_1,q_2)$ given above, we obtain a
$\C[B_n]$-module structure on $\bE$. This is essentially the Burau
representation; it differs from the standard definition
\cites{Burau,Jones:87,BLM} by a simple change of parameters (see
\cite{DG}).

\begin{rmk}\label{rmk:specializations}
Set $q = -q_2/q_1$.  There are well-known algebra isomorphisms \[
H_n(q_1,q_2) \cong H_n(-1,q),\quad H_n(q_1,q_2) \cong H_n(1,-q)\]
defined by sending $T_i \mapsto -q_1T_i$, $T_i \mapsto q_1T_i$
respectively.  Moreover, the map $T_i \mapsto q^{-1/2}T_i$ defines an
algebra isomorphism \[H_n(-q^{-1/2}, q^{1/2}) \cong H_n(-1,q),\] so
$H_n(-q^{-1/2}, q^{1/2}) \cong H_n(q_1,q_2)$. The algebra
$H_n(-q^{-1/2}, q^{1/2})$, the ``balanced'' form of the Iwahori--Hecke
algebra, is often preferred in the theory of quantum groups. We choose
to work with the generalized Burau representation because it makes
sense in general, including in all the one-parameter versions of
$H_n(q_1,q_2)$.
\end{rmk}

\section{The decomposition $\bE = \bL \oplus \bF$}\label{sec:Hecke}
\noindent
By direct computation, we notice the following explicit eigenvectors
for the $\ov{T}_i$ operators defined in the previous section.

\begin{lem}\label{lem:eigen}
  Assume that $q_1q_2 \ne 0$ and set $q= -q_2/q_1$. For any $i = 1,
  \dots, n-1$ the operator $\ov{T}_i$ has eigenvectors:
  \begin{enumerate}
    \item $\ee_1, \dots, \ee_{i-1}, \ee_i+\ee_{i+1}, \ee_{i+2}, \dots,
      \ee_n$ with eigenvalue $q_1$.
    
  \item $q\ee_i - \ee_{i+1}$ with eigenvalue $q_2$.
  \end{enumerate}
  In particular, $\ov{T}_i$ is diagonalizable if and only if $q_1 \ne
  q_2$. Moreover, the vector $\bl_0 := \ee_1+\cdots+\ee_n$ is a
  simultaneous eigenvector for all the $\ov{T}_i$.
\end{lem}

\begin{proof}
Parts (a), (b) are easily checked. Observe that $\ee_i+\ee_{i+1}$ and
$q\ee_i - \ee_{i+1}$ are linearly dependent if and only if $q= -1$,
which proves the diagonalizability claim, as $q= -1 \iff q_1=q_2$.
\end{proof}

\begin{rmk}
If $q_1=q_2$ then the $\ov{T}_i$ have only one eigenvalue and the
corresponding eigenspace has dimension $n-1$. 
\end{rmk}

As in \cite{DG}*{\S3} we consider the following
$H_n(q_1,q_2)$-submodules of $\bE$:
\[
\bL = \C \bl_0 = \C(\ee_1 + \cdots + \ee_n), \quad \bF =
\textstyle\bigoplus_{i=1}^{n-1} \C (q\ee_i - \ee_{i+1}).
\]
Since $q \ne 0$, the spanning vectors $q\ee_i - \ee_{i+1}$ are
linearly independent, so $\dim \bF = n-1$.

We aim now to show that $\bE=\bL \oplus \bF$ under suitable
hypotheses.  The following result is standard; see e.g.,
\cite{Mathas}*{Exercise 1.4}.

\begin{prop}\label{prop:splitting}
  Suppose that $q_1q_2 \ne 0$. Set $q = -q_2/q_1$ and
  $[n]_q = 1 + q + \cdots + q^{n-1}$.
  \begin{enumerate}
  \item $\bE = \bL \oplus \bF$ if and only if $[n]_q \ne 0$. 
  \item If $n > 2$ then $\bF$ is irreducible as an
    $H_n(q_1,q_2)$-module if and only if $[n]_q \ne 0$. (If $n=2$ then
    $\bF$ is irreducible for any $q$.)
  \end{enumerate}
\end{prop}

\begin{proof}
For (a), observe that the determinant of the matrix with the columns
$q\ee_1-\ee_2, \dots, q\ee_{n-1}-\ee_n, \bl_0$ is equal to $[n]_q$.
For (b), a direct argument can be found in \cite{DG}.
\end{proof}

\begin{rmk}
As $[n]_1 = n$, the decomposition $\bE = \bL \oplus \bF$ holds at $q=1$.
\end{rmk}

\section{Orthogonal group}\noindent
Recall that we always assume that $q_1q_2 \ne 0$ (hence $q \ne 0$). We
will need the nondegenerate symmetric bilinear form $\bil{-}{-}$ on
$\bE$ defined by the rule
\begin{equation}
  \bil{\ee_i}{\ee_j} = \delta_{ij} q^{j-1}
\end{equation}
extended bilinearly, where $q = -q_2/q_1$.  Let $J = \text{diag}(1, q,
\dots, q^{n-1})$ be the matrix of the form with respect to the
$\{\ee_i\}$-basis.

Observe that $\bil{q\ee_i - \ee_{i+1}}{\bl_0} = 0$ for all $i=1,
\dots, n-1$. Thus $\bF \subset \bL^\perp$ (the orthogonal complement
with respect to the form). Since $\dim \bL^\perp = n-1$ by the
standard theory of bilinear forms, it follows by dimension comparison
that $\bF = \bL^\perp$.

Now we consider certain orthogonal operators on $\bE$. Assume from now
on that $q_1 \ne q_2$. (This is equivalent to assuming that $q \ne
-1$.) Then we may define elements $S_i \in H_n(q_1,q_2)$ by
\begin{equation}
S_i = \tfrac{1}{q_1-q_2} \big( 2T_i - (q_1+q_2) \big)
\end{equation}
for $i = 1, \dots, n-1$. A simple calculation with the defining
quadratic relation in $H_n(q_1,q_2)$ shows that
\begin{equation}
  S_i^2=1.  
\end{equation}
Let $\ov{S}_i \in \End(\bE)$ be the corresponding linear
operator, defined by replacing $T_i$ by its image $\ov{T}_i$. The
$\ov{T}_i$-eigenvectors are also $\ov{S}_i$-eigenvectors, and the
$\ov{T}_i$-eigenvalues $q_1,q_2$ have been ``deformed'' to
$\ov{S}_i$-eigenvalues $1,-1$ respectively.

Recall \cite{Bour}*{Ch.~5, \S2, Nos.~1--2} that a linear endomorphism $s$
in $\End(\bE)$ is:
\begin{itemize}
\item a \emph{pseudo-reflection} if $1-s$ has rank 1.
\item a \emph{reflection} if, additionally, $s^2=1$. 
\end{itemize}
Lemma \ref{lem:eigen} and the fact that $S_i^2=1$ implies that
$\ov{S}_i$ is a reflection, so the group generated by the $\ov{S}_i$
is a reflection group.

Let $\OO(\bE)$ be the orthogonal group of operators $S \in \End(\bE)$
preserving the bilinear form $\bil{-}{-}$, in the sense that
$\bil{Sv}{Sw} = \bil{v}{w}$ for all $v,w \in \bE$.

\begin{lem}\label{lem:S_i}
Assume that $q_1 \ne q_2$ (equivalently, $q \ne -1$). 
\begin{enumerate}
\item $S_i \cdot \bl_0 = \bl_0$.

\item $S_i$ belongs to $\OO(\bE)$; i.e., $\ov{S}_i^\transpose J \ov{S}_i
  = J$.

\item $H_n(q_1,q_2)$ is generated, as an algebra, by the $S_i$.
\end{enumerate}
\end{lem}

\begin{proof}
(a) follows immediately from Lemma \ref{lem:eigen}, as $\bl_0$ is a
  sum of $S_i$-fixed points.

(b) By Lemma \ref{lem:eigen}, $\bE = \bE_1 \oplus \bE_{-1}$, where
  $\bE_1$, $\bE_{-1}$ are the eigenspaces belonging to the
  $S_i$-eigenvalues $1,-1$ respectively. By definition, the $\ee_i$
  are pairwise orthogonal with respect to the form. Notice that the
  eigenvectors $\ee_i+\ee_{i+1}$ and $q\ee_i-\ee_{i+1}$ are also
  orthogonal. It follows that $\bE_1 \perp \bE_{-1}$. Given any $v,w
  \in \bE$, write $v = v_1 + v_{-1}$, $w = w_1 + w_{-1}$ (uniquely)
  where $v_1,w_1 \in \bE_1$, $v_{-1},w_{-1} \in \bE_{-1}$. Then
  \begin{align*}
  \bil{S_i \cdot v}{S_i \cdot w} &= \bil{v_1-v_{-1}}{w_1-w_{-1}}\\ &=
  \bil{v_1}{w_1} + \bil{v_{-1}}{w_{-1}}\\ &=
  \bil{v_1+v_{-1}}{w_1+w_{-1}} = \bil{v}{w}.
  \end{align*}
  Thus $S_i$ preserves the form, which implies the result.

(c) This is immediate from the fact that the mapping \[T_i \mapsto
  \tfrac{1}{q_1-q_2} \big( 2T_i - (q_1+q_2) \big) = S_i\] is
  invertible, with inverse given by \[S_i \mapsto \tfrac{q_1-q_2}{2}
  S_i + \tfrac{q_1+q_2}{2} = T_i.\] So any linear combination of
  products of $S_i$'s is expressible as a linear combination of
  products of $T_i$'s, and vice versa.
\end{proof}

Now define $\ee'_i = q^{-(i-1)/2} \ee_i$. Then the basis $\{\ee'_1,
\dots, \ee'_n\}$ is orthonormal with respect to the form; that is,
\[
\bil{\ee'_i}{\ee'_j} = \delta_{ij}
\]
where $\delta_{ij}$ is the usual Kronecker delta function.
For each $i = 1, \dots, n-1$ we define
\[
\ff_i = \sqrt{q}\,\ee'_i - \ee'_{i+1}. 
\]
Notice that $\ff_i$ is a nonzero scalar multiple of the
$S_i$-eigenvector $q \ee_i - \ee_{i+1}$ in Lemma \ref{lem:eigen},
hence is itself an $S_i$-eigenvector (of eigenvalue $-1$). 

\begin{lem}\label{lem:reflection}
Assume that $q \ne -1$. Then $\ov{S}_i$ is the reflection in the
complex hyperplane $H_i = \ff_i^\perp$, so
\[
S_i \cdot v = v - 2 \proj{\ff_i}{v} 
\]
for any $v \in \bE$. In particular,
\begin{enumerate}
\item $S_i \cdot \ee'_i = \ee'_i - \frac{2\sqrt{q}}{1+q} \ff_i =
  \tfrac{1-q}{1+q} \ee'_i + \tfrac{2\sqrt{q}}{1+q} \ee'_{i+1}$.

\item $S_i \cdot \ee'_{i+1} = \ee'_{i+1} + \frac{2}{1+q} \ff_i =
  \tfrac{2\sqrt{q}}{1+q} \ee'_i - \tfrac{1-q}{1+q} \ee'_{i+1}$.

\item $S_i \cdot \ee'_j = \ee'_j$ for all $j \ne i, i+1$.
\end{enumerate}
\end{lem}

\begin{proof}
The displayed formula is standard \cite{Bour}*{Ch.~5, \S3}. It can
also be verified by direct computation from the definition of $S_i$
and equations \eqref{eq:Burau}.

Part (c) is immediate once one notices that the $\ee'_j$ for $j \ne i,
i+1$ belong to the $1$-eigenspace of $S_i$. Moreover, we have
\[
\bil{\ff_i}{\ff_i} = 1+q,\quad \bil{\ee'_i}{\ff_i} = \sqrt{q}, \quad
\bil{\ee'_{i+1}}{\ff_i} = - 1
\]
which gives the formulas in (a), (b).
\end{proof}

For concreteness, the action in Lemma \ref{lem:reflection} of $S_i$ on
the orthonormal basis $\ee'_1, \dots, \ee'_n$ is given by the $n
\times n$ block matrix
\[
\ov{S}_i = 
\begin{bmatrix}
  I_{i-1} & 0 & 0\\
  0 & Q & 0 \\
  0 & 0 & I_{n-1-i}
\end{bmatrix}
\text{ where } Q =
\frac{1}{1+q}
\begin{bmatrix}
  1-q & 2\sqrt{q} \\
  2\sqrt{q} & -(1-q)
\end{bmatrix}
\]
where $I_k$ is the $k \times k$ identity matrix.  Notice that the
action of the $S_i$ depends only on $q = -q_2/q_1$ and $Q$ is a $2
\times 2$ orthogonal matrix, in the sense that $Q^\transpose Q$ is the
$2 \times 2$ identity matrix.

\section{Twin group}\label{sec:twin}\noindent
Let $W_n$ be the Weyl group of $\GL(\bE)$, which we identify with the
set of $n \times n$ permutation matrices. Then $W_n \cong \Sym_n$, the
symmetric group on $n$ letters, generated by $s_1, \dots, s_{n-1}$
subject to the standard Coxeter relations: $s_i^2=1$, $s_is_{i+1}s_i =
s_{i+1}s_is_{i+1}$, and $s_is_j = s_js_i$ if $i \ne j$.

Let $TW_n$ be the \emph{twin group} \cite{Khovanov} on $n$ strands;
that is, the group generated by $t_1, \dots, t_{n-1}$ subject to the
defining relations
\begin{equation}\label{eq:twin-rels}
t_i^2 = 1, \quad t_i t_j = t_j t_i \text{ if } |i-j| > 1. 
\end{equation}
We have a quotient mapping $TW_n \twoheadrightarrow W_n$ (with kernel
the subgroup generated by all $(t_it_{i+1})^3$) defined by $t_i
\mapsto s_i$. 

Since the $\ov{S}_i$ satisfy the defining relations of $TW_n$,
the linear mapping
\[
\rho: TW_n \to \GL(\bE) \text{ defined by } t_i \mapsto \ov{S}_i
\]
is a representation. This is the analogue of the Burau representation
adapted to the twin group.

\begin{rmk}
Unless $q=1$, the operators $\ov{S}_i$ do not satisfy the braid
relations.  A calculation from the definition of $S_i$ reveals that
\[
S_iS_{i+1}S_i - S_{i+1}S_iS_{i+1} = -\tfrac{(1-q)^2}{(1+q)^2}\, (S_i -
S_{i+1})
\]
holds in $H_n(q_1,q_2)$. In particular, $\ov{S}_i\ov{S}_{i+1}\ov{S}_i
= \ov{S}_{i+1}\ov{S}_i\ov{S}_{i+1}$ if and only if $q=1$. Thus the
representation $\rho$ factors through $W_n$ at $q=1$, giving the
natural $n$-dimensional permutation representation of $W_n$; at $q \ne
1$ we have deformed the natural representation away from $W_n$ to a
representation of $TW_n$.
\end{rmk}

The linear extension of $\rho$ to the group algebra $\C[TW_n]$ factors
through $H_n(q_1,q_2)$, via $t_i \mapsto S_i \mapsto
\ov{S}_i$. As each $\ov{S}_i$ belongs to $\OO(\bE)$ it follows that the
image of the representation $\rho$ is contained in $\OO(\bE)$; in other
words, $\rho$ is an orthogonal representation of $TW_n$. Note that
\[
\OO(\bE) \cong \OO_n(\C) = \{A \in \GL_n(\C): A^\transpose A = I\}
\]
when the operators are expressed as matrices with respect to the
orthonormal basis $\{\ee'_1, \dots, \ee'_n\}$.

Since all the $\rho(t_i) = \ov{S}_i$ fix $\bl_0$ by Lemma \ref{lem:S_i},
the line $\bL$ is isomorphic to $\C$, the trivial $\C[TW_n]$-module.

\begin{lem}\label{lem:H-structure}
  Suppose that $q \ne -1$. Any $H_n(q_1,q_2)$-module $V$ becomes a
  $\C[TW_n]$-module by defining $t_i \cdot v = S_i \cdot v$, for any
  $v \in V$. Its submodule structure is the same, regarded as a module
  for either algebra.
\end{lem}

\begin{proof}
This follows from Lemma \ref{lem:S_i}(c) and the fact that the linear
transformation (on $H_n(q_1,q_2)$) defined by $T_i \mapsto S_i$ is
invertible.
\end{proof}

Assume henceforth that $[n]_q \ne 0$. Combining Proposition
\ref{prop:splitting} and Lemma \ref{lem:H-structure}, we conclude that
\begin{equation}
\bE = \bF \oplus \bL
\end{equation}
as $\C[TW_n]$-modules, where $\bL$ and $\bF$ are irreducible
$\C[TW_n]$-submodules. Hence
\begin{equation}
  \rho(TW_n) \subset \OO(\bL) \times \OO(\bF)
\end{equation}
where $ \OO(\bL)$, $\OO(\bF)$ are taken with respect to the restriction of
the bilinear form.

\begin{lem}\label{lem:action-on-f-basis}
Suppose that $q \ne -1$ and $[n]_q \ne 0$.  The action of the $S_i$ on
$\bF$ is given by
\[
S_i \cdot w = w -  2 \proj{\ff_i}{w} 
\]
for any $w \in \bF$. Hence, $S_i$ acts on $\bF$ as reflection in the
hyperplane $\ff_i^\perp = \{w \in \bF: \bil{\ff_i}{w}=0\}$.  In
particular,
\begin{enumerate}
\item $S_i \cdot \ff_{i-1} = \ff_{i-1} + \frac{2\sqrt{q}}{1+q} \ff_i$.
\item $S_i \cdot \ff_i = - \ff_i$.
\item $S_i \cdot \ff_{i+1} = \ff_{i+1} + \frac{2\sqrt{q}}{1+q} \ff_i$.
\item $S_i \cdot \ff_j = \ff_j$ for all $j \ne i-1, i, i+1$.
\end{enumerate}
Hence the $S_i$-eigenspace of eigenvalue $-1$ is spanned by
$\ff_i$, and the $S_i$-eigenspace of eigenvalue $1$ is spanned by
\[
\ff_{i-1} + \tfrac{\sqrt{q}}{1+q} \ff_i,\ \ff_{i+1} +
\tfrac{\sqrt{q}}{1+q} \ff_i, \text{ and } \ff_j \text{ for all } j
\ne i-1, i, i+1.
\]
\end{lem}

\begin{proof}
Given $v \in \bE$, there exist unique $u\in \bL$, $w \in \bF$ such
that $v = u+w$. As $S_i \cdot u = u$, Lemma \ref{lem:S_i} says that
\[
S_i \cdot (u+w) = u + S_i  \cdot w .
\]
On the other hand, Lemma \ref{lem:reflection} in light of the equality
$\bil{\ff_i}{u} = 0$ says that
\[
S_i \cdot (u+w) = u + w - 2 \proj{\ff_i}{u+w} = 
u + w - 2\proj{\ff_i}{w} 
\]
and the first claim follows by comparing the right hand sides of the
two displayed equalities.  Formulas (a)--(d) then follow
immediately. (They also follow from Lemma \ref{lem:reflection} by a
routine calculation.) They in turn imply the claims about the
eigenvectors, which implies the final claim.
\end{proof}

Since (under the hypotheses of the lemma) the $S_i$ act as reflections
on $\bF$, we call it the \emph{reduced} reflection representation.  We
record the following observation for later use.

\begin{lem}\label{lem:projection}
Suppose that $q \ne -1$ and $[n]_q \ne 0$.  The matrix of the
orthogonal projection $\bE \twoheadrightarrow \bL$ with respect to the
orthonormal basis $\ee'_1, \dots, \ee'_n$ is
\[
\frac{1}{[n]_q} \bigg(q^{(i+j-2)/2}\bigg)_{i,j=1,\dots, n} 
\]
and this projection is an endomorphism of $\bE$ commuting with the
action of $\OO(\bE)$ and thus also with the action of $TW_n$.
\end{lem}

\begin{proof}
Let $P$ be the projection operator. Then we have
\[
P(\ee'_j) = \proj{\bl_0}{\ee'_j} = \frac{q^{(j-1)/2}}{[n]_q} \bl_0 =
\frac{q^{(j-1)/2}}{[n]_q} \sum_{i=1}^n q^{(i-1)/2} \ee'_i.
\]
This proves the first claim, and the rest follows from Lemma
\ref{lem:H-structure} and the definitions of the actions.
\end{proof}

\section{Density}\label{sec:density}\noindent
The main result of this section is the following density result, which
is the central technical result of the paper. Recall the ``quantum
factorial'' notation
\[
[n]_q^! = [1]_q [2]_q \cdots [n]_q
\]
for any positive integer $n$. 

\begin{thm}\label{thm:density}
Let $\rho: TW_n \to \OO(\bE)$ be the reflection representation, defined
on generators by sending $t_i$ to $\ov{S}_i$ for all $i = 1, \dots,
n-1$. Assume that $[n]_q \ne 0$, $[n-2]^!_q \ne 0$, and
\[
  q \ne \frac{-\lambda \pm \sqrt{-1-2\lambda}}{1+\lambda}
\]
for any $\lambda = \cos(2k\pi/m)$ with $m \in \Z_{\ge 0}$. Then the
image $\rho(TW_n)$ is Zariski-dense in $\OO(\bL) \times \OO(\bF)$. Hence
the image of the reduced reflection representation $TW_n \to \OO(\bF)$
is Zariski-dense in $\OO(\bF)$.
\end{thm}

\begin{rmk}\label{rmk:density}
  Set $w^\pm(\lambda) = (-\lambda \pm
  \sqrt{-1-2\lambda})/(1+\lambda)$. As $w^+(\lambda) w^-(\lambda) =
  1$, we have $w^-(\lambda) = 1/w^+(\lambda)$.  If
  $\lambda=\cos(2k\pi/m)$ for some $k,m$ then $\lambda$ is a real
  number satisfying $-1 \le \lambda \le 1$. There are two cases to
  consider:
  \begin{enumerate}
  \item [(i)] If $-1 < \lambda \le -\frac{1}{2}$ then $w^+(\lambda)$
    is real and satisfies $w^+(\lambda) \ge 1$. Also $w^+(\lambda) \to
    \infty$ as $\lambda \to -1$. In this case $w^-(\lambda) =
    1/w^+(\lambda)$ is also real and satisfies $0 < w^-(\lambda) \le
    1$.

  \item[(ii)] If $-\frac{1}{2} < \lambda \le 1$ then $w^+(\lambda)$,
    $w^-(\lambda)$ are non-real conjugate complex values, so both lie
    on the unit circle $|z|=1$ in the complex plane.
  \end{enumerate}  
  In particular, all the hypotheses of Theorem \ref{thm:density} hold
  if $q$ is chosen to be any complex number not on the positive real
  axis or the unit circle.
\end{rmk}

The proof of Theorem \ref{thm:density} will take up the rest of this
section. Here is an outline of the strategy. We know that
$\rho(TW_n)\subset \OO(\bL) \times \OO(\bF)$. Since $\OO(\bL)$ is finite,
its Lie algebra is zero and so we have
\[
\Lie \OO(\bE) = \Lie \OO(\bF) \cong \mathfrak{so}_{n-1}(\C). 
\]
Hence the dimension of the Lie algebra of $\OO(\bE)$ is $\dim
\mathfrak{so}_{n-1}(\C) = \binom{n-1}{2}$. Put
\[
G = \rho(TW_n), \quad \ov{G} = \text{ the Zariski-closure of $G$ in
  $\OO(\bF)$.}
\]
Under the stated assumptions on $q$, we will use the defining matrices
$ \ov{S}_i$ to find $\binom{n-1}{2}$ linearly independent elements of
$\Lie \ov{G}$. This forces $\Lie \ov{G} = \mathfrak{so}_{n-1}(\C)$ and
hence justifies the conclusion $\ov{G} = \OO(\bF)$.

For $i=1, \dots, n-1$, recall from Lemma \ref{lem:reflection} that
$\ov{S}_i = \rho(t_i)$ is given (with respect to the orthonormal basis
$\ee'_1, \dots, \ee'_n$) by the block diagonal matrix
\[
\ov{S}_i = \text{diag}(I_{i-1}, Q, I_{n-1-i}), \quad\text{where } Q =
\begin{bmatrix}
  a & b \\
  b & -a
\end{bmatrix} 
\]
and $a = (1-q)/(1+q)$, $b = 2\sqrt{q}/(1+q)$. Notice that $a^2+b^2 =
1$. Since $\ov{S}_i$ squares to 1 and $\ov{S}_i\ov{S}_j =
\ov{S}_j\ov{S}_i$ for all $|i-j|>1$, understanding $G = \rho(TW_n)$
requires studying the products
\[
\ov{S}_{i,i+1} := \ov{S}_i \ov{S}_{i+1} \;\text{ for } i = 1, \dots, n-2. 
\]
A direct calculation shows that
\[
\ov{S}_{i,i+1}=
\begin{bmatrix}
  I_{i-1} & 0 & 0\\
  0 & R & 0 \\
  0 & 0 & I_{n-2-i}
\end{bmatrix}
\text{ where } R =
\begin{bmatrix}
 a & ab & b^2  \\
 b  & -a^2 &-ab \\
 0 & b  & -a
\end{bmatrix}.
\]
We need to study the powers of this matrix.

The next step arises from classical geometry in three-dimensional
euclidean spaces. Assume for the moment that $q>0$ is a real number
and that our vector space $\bE = \bL \oplus \bF$ is over the real
numbers. We restrict our attention to the three-dimensional subspace
\[
V = \R\ee'_i \oplus \R\ee'_{i+1} \oplus \R\ee'_{i+2} \cong \R^3.
\]
We know that $\ov{S}_i$ is reflection in $\ff_i^\perp$ for each
$i$. Let $\theta/2$ be the angle between $\ff_i$ and $\ff_{i+1}$ and
let $\mathbf{n}$ be the unit vector in the direction of the
cross-product $\ff_i \times \ff_{i+1}$. Then $\ov{S}_{i,i+1}$ is the
composition of the two reflections, hence a rotation. In fact, it is
the rotation in the plane orthogonal to $\mathbf{n}$. By the classical
Rodrigue\'{s} rotation formula \cite{Rod}, we have
\[
R = I + (\sin \theta)N + (1-\cos \theta)N^2
\]
where $I$ is the identity matrix, $\mathbf{n} = (n_1,n_2,n_3)$ in local
coordinates of $V \cong \R^3$, and where
\[
N=
\begin{bmatrix}
  0&-n_3&n_2\\
  n_3&0&-n_1\\
  -n_2&n_1&0
\end{bmatrix}
\]
is the ``cross-product'' matrix representing the linear map $\mathbf{x}
\mapsto \mathbf{n} \times \mathbf{x}$ on $V \cong \R^3$. Since $\mathbf{n}$
is proportional to the cross-product
\[
\ff_i \times \ff_{i+1} = (1,\sqrt{q},q),
\]
after normalization we get $\mathbf{n} = [3]_q^{-1/2}
(1,\sqrt{q},q)$ and hence
\[
N= \frac{1}{\sqrt{[3]_q}}\, 
\begin{bmatrix}
  0&-q&\sqrt{q}\\
  q&0&-1\\
  -\sqrt{q}&1&0
\end{bmatrix}.
\]
It is well known that $N \in \mathfrak{so}(3) = \mathfrak{so}_3(\R)$
and that $R = \exp(\theta N)$. Furthermore, $N$ satisfies the relation
$N^3 = -N$.

Now we revert to the general case $q \in \C$. Motivated by the above
considerations, for each $i = 1, \dots, n-2$, define
\[
N_i = \text{diag}(0_{i-1}, N, 0_{n-2-i})
\]
where $0_k$ denotes a $k \times k$ zero matrix. Then $N_i \in
\mathfrak{so}_{n}(\C)$.  Consider the one-parameter subgroup $G_i$
consisting of all matrices of the form
\begin{equation}
  R(\alpha) = \exp(\alpha N_i) = I + (\sin \alpha)N_i + (1-\cos
  \alpha)N_i^2 \quad (\alpha \in \C).
\end{equation}
Note that $R(0)=I$ is the identity matrix; more generally
\begin{equation}
R(2k\pi) = I \quad \text{for all } k\in \Z.
\end{equation}
The complex sine function is surjective, so for any $z \in \C$, there
exists an $\alpha \in \C$ such that $z = \sin \alpha$. As $\sin^2
\alpha + \cos^2 \alpha = 1$, it follows that $\cos z = \sqrt{1-z^2}$
for the appropriate choice of the square root. Hence, for any $z \in
\C$ there is an $\alpha \in \C$ for which
\begin{equation}
R(\alpha) = I +zN_i+(1-\sqrt{1-z^2})N_i^2 .
\end{equation}
In this sense, $R(\alpha)$ depends not on $\alpha$ but only on $z =
\sin \alpha$ and a choice of the square root. More precisely, we have
\[
\alpha = -\sqrt{-1} \log(z\sqrt{-1} + \sqrt{1-z^2})
\]
(where $\sqrt{-1}$ is the imaginary unit) and we have already chosen
the square root of $1-z^2$ in choosing $\alpha$.

%% \begin{lem}
%% Let $G_i$ be the set of all matrices of the form
%% \[
%% R(z) = I_{n}+z\, N_i+(1-\sqrt{1-z^2})\,N_i^2 \quad (z \in \C).
%% \]
%% Then $G_i$ is a one-parameter subgroup of $\SO_{n}(\C)$ containing
%% the identity matrix $I_{n}$.
%% \end{lem}

%% \begin{proof}
%% The proof that $G_i$ is a subgroup follows from a direct calculation
%% using the relation $N_i^3=-N_i$, the identity
%% \[
%% R(z_1)R(z_2) = R\left(z_1 \sqrt{1-z_2^2} + z_2 \sqrt{1-z_1^2}\right),
%% \]
%% and the fact that $R(0)=I_n$. In particular, $R(z)^{-1} = R(-z)$. Note
%% that the formulas are not at all surprising if we formally regard $z_1
%% = \sin \theta_1$, $z_2 = \sin \theta_2$ based on the analogy with
%% euclidean geometry, as the input $z_1 \sqrt{1-z_2^2} + z_2
%% \sqrt{1-z_1^2}$ on the right hand side of the displayed equation
%% corresponds to the addition formula for sine.
%% \end{proof}

\begin{lem}
Fix any $i=1,\ldots, n-2$. Then $\ov{S}_{i,i+1} \in G_i$. Hence, the
cyclic subgroup generated by $\ov{S}_{i,i+1}$ is contained in $G_i$.
\end{lem}

\begin{proof}
The assertion holds if and only if there exists a value $z \in \C$
such that
\[
R=I_3+ z\, N + (1-\sqrt{1-z^2})\,N^2.
\]
A calculation shows that
\[
z=\frac{2\sqrt{q\,[3]_q}}{[2]_q^2}\text{ and }
\sqrt{1-z^2}= \frac{-[2]_{q^{2}}}{[2]_q^2}
\]
gives a solution, which corresponds to
\[
\alpha = -\sqrt{-1}\log\bigg(\frac{2\sqrt{-q\,[3]_q} -
  [2]_{q^{2}}}{[2]_q^2} \bigg).
\]
This proves the claim.
\end{proof}

\begin{lem}
Fix any index $i$ in the range $1, \dots, n-2$. For $\alpha$ as
above, we have
\[
\ov{S}_{i,i+1}^{\;k} = R(\alpha)^k = I + z U_{k-1}(\cos \alpha)\, N
  + (1 - T_k(\cos \alpha))\, N^2
\]
for all $k \ge 0$, where $T_k$, $U_k$ are the Chebyshev polynomials of
the first and second kind, respectively.
\end{lem}

\begin{proof}
This follows from the well known properties of Chebyshev polynomials.
As $\ov{S}_{i,i+1} = R(\alpha)$ we have
\[
\ov{S}_{i,i+1}^{\;k} = R(\alpha)^k = \exp(k\alpha N) =
I + \sin(k\alpha) N + (1-\cos(k\alpha))N^2
\]
and the result follows from the standard trigonometric definition of
Chebyshev polynomials (see e.g.~\cite{Cheby}), which defines the
polynomials $U_{k-1}$, $T_k$ in order that $(\sin \alpha) U_{k-1}(\cos
\alpha) = \sin(k\alpha)$ and $T_k(\cos \alpha) = \cos(k\alpha)$ for
all nonnegative integers $k$.
\end{proof}

\begin{cor}
Fix any index $i$ in the range $1, \dots, n-2$. With $z=\sin \alpha =
2\sqrt{q\,[3]_q}/[2]_q^2$ as above, $\ov{S}_{i,i+1}^{\;k} = I$
if and only if
\[
U_{k-1}(\cos \alpha) = 0 \text{ and } T_{k}(\cos \alpha) = 1,
\]
which happens if and only if $\alpha = 2l\pi/k$ for some $l \in \{0, 1,
\dots, k-1\}$.
\end{cor}

\begin{proof}
The joint equations $U_{k-1}(\cos \alpha) = 0$, $T_{k}(\cos \alpha) =
1$ are equivalent to the conditions $\sin(k\alpha) = 0$,
$\cos(k\alpha) = 1$, which in turn are equivalent to the single
equality $\exp(k\alpha \sqrt{-1}) = 1$. The result follows.
\end{proof}

\begin{rmk}
For all $i$, the last proof shows that $\ov{S}_{i,i+1} = R(\alpha)$
has order $k$ if and only if $\alpha$ is an integer multiple of
$2\pi/k$.
\end{rmk}

\begin{cor}\label{cor:finite}
  The matrix $\ov{S}_{i,i+1} = R(\alpha)$ has finite order if and only if
  \[
  q = \frac{-\lambda \pm \sqrt{-1-2\lambda}}{1+\lambda}
  \]
  for some $\lambda = \cos(2l\pi/k)$ with $l \in \{0,1, \dots, k-1\}$. 
\end{cor}

\begin{proof}
Finite order occurs if and only if $\alpha = 2l\pi/k$ with $l
\in \{0,1, \dots, k-1\}$. For any such $\alpha$, we have $\cos \alpha
= -[2]_{q^2}/[2]_q^2 = -(1+q^2)/(1+q)^2$. Since $z = \sin \alpha$,
\[
\sqrt{1 - z^2} = \cos \alpha = -(1+q^2)/(1+q)^2
\]
as previously noted. Setting $\lambda = \cos \alpha = \sqrt{1 - z^2}$
in the above and solving for $q$ gives the result.
\end{proof}

\begin{rmk}
An interesting question is to characterize the reflection group
generated by the $\ov{S}_i$ in case $q$ has one of the values in
Corollary \ref{cor:finite}. We do not know whether the group is
finite, for example, except when $n=3$, in which case it is a finite
dihedral group.
\end{rmk}

For all integers $i,j$ with $1\leq i < j\leq n-1$, define elements
$L_{i,j}$ in $\Lie(\overline{\rho(TW_n)})$ as follows. First set
$L_{i,i+1}=\sqrt{[3]_q}\,K_i$. The scaling factor $\sqrt{[3]_q}$ is
present for convenience to clear denominators. Next, inductively
define $L_{i,j+1}=[L_{i,j},L_{j,j+1}]$ for $j=i+1,\ldots, n-2$.  For
all $1\leq i < j \leq n-1$ set $\overline{e}_{i,j} = e_{i,j}-e_{j,i}$,
where $e_{i,j}$ is the matrix unit with a $1$ in row $i$, column $j$
and $0$ elsewhere.

\begin{lem}\label{lem:Lrs}
The Lie algebra elements $L_{r,s}$ are given by the formula
\begin{multline*}
  L_{r,s}= (-1)^t \Big(
  \Big. q^{(t-1)/2}\,[t-1]_q\,\overline{e}_{r,s-1}+
  q^t\,\overline{e}_{r,s}-q^{1/2}\,[t]_q\,\overline{e}_{r,s+1}
  \\ -\sum_{i=1}^{t-2}
  q^{(i+1)/2}\,\overline{e}_{r+i,s-1}+\sum_{j=1}^t
  q^{(j-1)/2}\,\overline{e}_{r+j,s+1} \Big. \Big)
\end{multline*}
where $t= s-r$.
\end{lem}

\begin{proof}
A straightforward induction on $t$ proves the formula. The base case,
$t = 1$, is the definition of the elements $L_{r,r+1}$ given
above. Then assuming the result for some $L_{r,s}$, the formula for
$L_{r,s+1}$ is derived from the definition
$L_{r,s+1}=[L_{r,s},L_{s,s+1}]$ by computing the bracket using the
matrix unit commutator formulas
$[e_{ij},e_{kl}]=\delta_{jk}e_{il}-\delta_{il}e_{ki}.$
\end{proof}

Next we investigate the independence of the elements $L_{r,s}$. It
will be helpful to consider a total order on the indexing set
\[
\Omega=\{(r,s)\,\mid \, 1\leq r <s \leq n-1\}.
\]
For $(i,j), (k,l)\in \Omega$, define $(i,j) \succ (k,l)$ if either $j
> l$ or both $j = l$ and $i<k$.  Note that the order $\succ$ first
compares column indices. Elements in a given column have greater
indices than all those in preceding columns, and elements in the same
column are ordered inversely by their row index.

\begin{lem}\label{lem:indep}
  Suppose that $[n-2]_q^! \ne 0$.  Then the elements $L_{r,s}$ with
  $1\leq r < s \leq n-1$ are linearly independent.
\end{lem}

\begin{proof}
Suppose that $\sum c_{r,s}L_{r,s} =0.$ We will show that each
$c_{r,s}=0$ starting with $c_{1,n-1}$, the term with subscript of
maximal order, and proceeding in decreasing order according to the
coefficient's subscript and ending at $c_{1,2}$.

Lemma \ref{lem:Lrs} implies that the coefficient of
$\overline{e}_{1,n}$ in the sum $\sum c_{r,s}L_{r,s} $ is, up to a
sign, equal to $\sqrt{q}[n-2]_q \,c_{1,n-1}$. As $q$ and $[n-2]_q$ are
assumed to be non-zero it follows that $c_{n-1,1}=0$. Next, consider
the coefficient of $\overline{e}_{2,n}$ in the sum. Since
$c_{1,n-1}=0$ the coefficient of $\overline{e}_{2,n}$ is, up to a
sign, $\sqrt{q}\,[n-3]_q\,c_{2,n-1}$. It follows that
$c_{2,n-2}=0$. Moving inductively through the subsequent indices in
$\Omega$ in descending order, from $(3,n-3)$ down to $(1,2)$
similarly shows that all $c_{r,s}=0$.
\end{proof}

\begin{rmk}
If $[j]_q=0$ for some $j=1,\ldots n-2$, then the explicit dependence
relation
\[
L_{1,j+1}+ \textstyle \sum_{i =1}^{j-1}
q^{(j-i)/2}\,\left(L_{1,j+1}+L_{2,j+1}\right) =0
\]
holds. We do not know whether the image $\rho(TW_n)$ is dense
in that case.
\end{rmk}

\begin{proof}[Proof of Theorem \ref{thm:density}]
Lemma \ref{lem:indep} implies Theorem \ref{thm:density} by a dimension
comparison, as the $L_{r,s}$ form a set of $\binom{n-1}{2}$
linearly independent elements in the Lie algebra $\Lie \OO(\bE) \cong
\mathfrak{so}_{n-1}(\C)$ (see the remarks at the beginning of this
section).
\end{proof}

\section{The partial Brauer algebra}\label{sec:PB}\noindent
Recall \cites{Brauer,HR} that the Brauer algebra $\B_r(\delta)$ with
parameter $\delta \in \C$ is the algebra with basis indexed by the
pairings of a finite set
\[
\mathbf{r} = \{1, \dots, r, 1', \dots, r'\}
\]
of $2r$ elements, where we define a pairing to be a set partition in
which all the subsets have cardinality two.  Each basis element is
typically pictured as an undirected graph on $2r$ vertices, arranged
in two rows with the numbers $1, \dots, r$ (resp., $1', \dots, r'$)
labeling the top (resp., bottom) row vertices in order from left to
right; an edge connects two vertices if and only if they are
paired. For example, the diagram
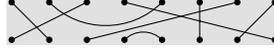
\[
\begin{minipage}{3.9cm}
\begin{tikzpicture}[scale=.5]
  \filldraw[fill= black!12,draw=black!12,line width=4pt] (1,0) rectangle (8,1);
    \foreach \x in {1,...,8} {
      % \draw (\x,0) circle (2pt);
      \filldraw [black] (\x, 0) circle (2pt);
      \filldraw [black] (\x, 1) circle (2pt);
    }
    \draw (1,1)--(2,0) (3,1)--(1,0);
    \draw (6,1)--(6,0) (4,1)--(8,0);
    \draw (3,0)--(7,1) (8,1)--(7,0);
    \draw (5,1) to[out=-135,in=-45] (2,1);
    \draw (4,0) to[out=45,in=135] (5,0);
\end{tikzpicture}
\end{minipage}
\]
depicts the pairing
\[
\{ \{1,2'\}, \{2,5\}, \{1',3\}, \{3',7\},
\{4,8'\}, \{4',5'\} \{6,6'\}, \{7',8\} \}.
\]
The product of two basis elements $d_1, d_2$ is obtained by stacking
$d_1$ above $d_2$, identifying (and removing) the middle rows of
vertices, and pairing two elements from the remaining two rows if and
only if there is a path between them. Closed loops in the middle rows
are also removed. If $d_1 \circ d_2$ is the resulting graph, the
product $d_1 d_2$ is defined by
\[
d_1 d_2 = \delta^N\, d_1 \circ d_2
\]
where $N$ is the number of such interior loops. This makes
$\B_r(\delta)$ into an associative algebra. If all edges in a diagram
connect one endpoint in the top row to one in the bottom row then the
diagram depicts a permutation; the subalgebra spanned by such diagrams
is isomorphic to the group algebra $\C[\Sym_r]$ of the symmetric group
$\Sym_r$.

Let $\PB_r(\delta,\delta')$ be the two-parameter ``partialization'' of
$\B_r(\delta)$ as defined in \cites{MM,HdM}; cf.\ also \cite{KM}. The
idea is to allow graphs with pairings and, possibly, a number of
isolated vertices. In other words, the subsets of the underlying set
partition have cardinality at most two. The resulting (larger) algebra
is the \emph{partial Brauer algebra}, also known as the ``rook Brauer
algebra''. The multiplication is again given by stacking, but a second
parameter $\delta'$ is introduced in order to track the number of
non-loops (isolated vertices or connected paths) in the deleted middle
row. Multiplication of basis elements is defined by
\begin{equation}
  d_1d_2 = \delta^{N_1} \delta'^{N_2} \, d_1 \circ d_2
\end{equation}
where $N_1$ and $N_2$ is the number of removed loops and non-loops,
respectively.  (The use of two parameters in this context goes back at
least to Mazorchuk \cite{Mazor}.) This makes $\PB_r(\delta,\delta')$
into an associative algebra. It contains $\B_r(\delta)$ as a
subalgebra (the span of the set of diagrams with no isolated points)
and also contains the partial permutation algebra
$\Ptn\Ptn_r(\delta')$, spanned by the partial diagrams in which no two
vertices in the same row are ever paired. The algebra
$\Ptn\Ptn_r(\delta')$ was studied in \cites{DG}; it is closely related
to the symmetric inverse semigroup (also known as the rook monoid)
studied by Munn \cites{Munn:57a,Munn:57b}, Solomon \cite{Solomon}, and
others. $\Ptn\Ptn_r(\delta')$ is the partialization of $\C[\Sym_r]$ in
the same sense that $\PB_r(\delta,\delta')$ is the partialization of
$\B_r(\delta)$.

If we set $\delta' = \delta$ then $\PB_r(\delta,\delta)$ is a
subalgebra of the partition algebra $\Ptn_r(\delta)$ of
\cites{Martin:book,Martin:94,Martin:96,Jones:94}.  The paper \cite{HR}
provides a convenient comprehensive summary of many basic properties
of $\Ptn_r(\delta)$.

\begin{thm}[\cites{KM,MM}]\label{thm:present}
Let $s_i$, $e_i$ (for $i = 1, \dots, r-1$) and $p_j$ (for $j = 1,
\dots, r$) be defined by
\begin{gather*}
s_i =\;
\begin{minipage}{3.9cm}
\begin{tikzpicture}[scale=.5]
  \filldraw[fill= black!12,draw=black!12,line width=4pt] (1,0) rectangle (8,1);
    \foreach \x in {1,3,4,5,6,8} {
      % \draw (\x,0) circle (2pt);
      \filldraw [black] (\x, 0) circle (2pt);
      \filldraw [black] (\x, 1) circle (2pt);
    }
    \draw (1,1)--(1,0) (3,1)--(3,0);
    \draw (6,1)--(6,0) (8,1)--(8,0);
    \draw (4,1)--(5,0) (5,1)--(4,0);
    \node at (2,0.5) {$\cdots$};
    \node at (7,0.5) {$\cdots$};
\end{tikzpicture}
\end{minipage} , \qquad
e_i =\;
\begin{minipage}{3.9cm}
\begin{tikzpicture}[scale=.5]
  \filldraw[fill= black!12,draw=black!12,line width=4pt] (1,0) rectangle (8,1);
    \foreach \x in {1,3,4,5,6,8} {
      % \draw (\x,0) circle (2pt);
      \filldraw [black] (\x, 0) circle (2pt);
      \filldraw [black] (\x, 1) circle (2pt);
    }
    \draw (1,1)--(1,0) (3,1)--(3,0);
    \draw (6,1)--(6,0) (8,1)--(8,0);
    \draw (5,1) to[out=-135,in=-45] (4,1);
    \draw (4,0) to[out=45,in=135] (5,0);
    \node at (2,0.5) {$\cdots$};
    \node at (7,0.5) {$\cdots$};
\end{tikzpicture}
\end{minipage}, \\
p_j =\;
\begin{minipage}{3.5cm}
\begin{tikzpicture}[scale=.5]
  \filldraw[fill= black!12,draw=black!12,line width=4pt] (1,0) rectangle (7,1);
    \foreach \x in {1,3,4,5,7} {
      % \draw (\x,0) circle (2pt);
      \filldraw [black] (\x, 0) circle (2pt);
      \filldraw [black] (\x, 1) circle (2pt);
    }
    \draw (1,1)--(1,0) (3,1)--(3,0);
    \draw (5,1)--(5,0) (7,1)--(7,0);
    \node at (2,0.5) {$\cdots$};
    \node at (6,0.5) {$\cdots$};
\end{tikzpicture}
\end{minipage}
\end{gather*}
depicting the pairings
\begin{gather*}
\{\{i, (i+1)'\}, \{i+1, i'\}\} \cup \{\{j,j'\}: j \ne i, i+1\},
\\ \{\{i, i+1\}, \{i', (i+1)'\}\} \cup \{\{j,j'\}: j \ne i, i+1\},
\end{gather*}
and the partial pairing
\[
\{\{j\}, \{j'\}\} \cup \{\{k,k'\}: k \ne j\}
\]
respectively. Then
\begin{enumerate}
\item The symmetric group algebra $\C[\Sym_r]$ is generated by the
  $s_i$ subject to the defining relations
  \begin{gather*}
    s_i^2 = 1, \ s_is_{i+1}s_i = s_{i+1}s_is_{i+1}, \
    s_is_j=s_js_i \text{ if } |i-j|>1.
  \end{gather*}
  
\item The Brauer algebra $\B_r(\delta)$ is generated by the
  $s_i$, $e_i$ for $i = 1, \dots, r-1$ subject to the
  defining relations 
  \begin{gather*}
    e_i^2 = \delta e_i, \
    e_ie_{i\pm 1}e_i = e_i, \ e_ie_j = e_je_i \text{ if } |i-j|>1, 
    \\ s_ie_i = e_is_i = e_i, \ s_ie_{i\pm 1}e_i = s_{i\pm 1}e_i,
    \ e_ie_{i\pm 1}s_i = e_is_{i\pm 1},\\  e_is_j=s_je_i
    \text{ if } |i-j|>1.
  \end{gather*}
  together with the relations in part (a).

\item The partial permutation algebra $\Ptn\Ptn_r(\delta')$ is
  generated by the $s_i$, $p_j$ for $i = 1, \dots, r-1$ and $j = 1,
  \dots, r$ subject to the defining relations
  \begin{gather*}
    p_i^2 = \delta' p_i, \ p_ip_j=p_jp_i \text{ if } i\ne j, 
    \ p_is_ip_i = p_ip_{i+1}, \\ s_ip_i=p_{i+1}s_i, \ s_ip_j = p_js_i
    \text{ if } j \ne i,i+1 
  \end{gather*}
  together with the relations in part (a).
  
\item The partial Brauer algebra $\PB_r(\delta, \delta')$ is generated
  by the $s_i$, $e_i$ for $i = 1, \dots, r-1$ along with the  $p_j$ for
  $j=1, \dots, r$ subject to the defining relations
  \begin{gather*}
    e_ip_ie_i = \delta' e_i, \ e_ip_ip_{i+1}=\delta' e_ip_i, 
    \ p_ip_{i+1}e_i=\delta' p_ie_i, \\
    p_ie_ip_i = p_ip_{i+1}, \ e_ip_i=e_ip_{i+1}, \ p_ie_i=p_{i+1}e_i,
    \\ e_ip_j=p_je_i \text{ if } j \ne i, i+1 %\\
    %s_{i+2}s_{i+1}e_ip_{i+2}p_{i+3} = s_is_{i+1}p_ie_ie_{i+2}p_{i+2}
  \end{gather*}
  along with the relations in parts (a), (b), and (c). 
\end{enumerate}
\end{thm}

\begin{rmk}
  (i) Part (a) of Theorem \ref{thm:present} is standard. Presentations
  similar to those in (b), (c) were known earlier
  \cites{Birman-Wenzl,Nazarov,Solomon}. A slightly different
  presentation of $\Ptn\Ptn_r(\delta')$ was derived in \cite{DG}; its
  equivalence with the presentation in part (c) is easy to check.  The
  book \cite{GM} gives a wealth of results on the semigroups related
  to the theorem.

  (ii) We use $e_i$ for cup-cap diagrams (as is common in the
  literature) and $p_i$ for projection diagrams. \emph{Warning:} those
  notations are transposed in \cite{MM}. The published version of
  \cite{MM}*{Prop.~5.2} includes a redundant relation (see
  \cite{KM}*{Lemma 5.2}) which we have omitted.

  (iii) As explained in \cite{HdM}, there are several other
  interesting subalgebras of $\PB_r(\delta,\delta')$; e.g., the
  Motzkin algebra \cite{BH}.
\end{rmk}

A special case of the following was observed in \cite{MM}*{\S6.2}; see
\cite{DG}*{Cor.~8.7} for a similar result in a different context. 

\begin{lem}
For any $\delta' \ne 0$, $\PB_r(\delta,\delta') \cong \PB_r(\delta,1)$
as algebras.
\end{lem}

\begin{proof}
The isomorphism is $\PB_r(\delta,\delta') \to \PB_r(\delta,1)$ is
given by mapping $p_j$ to $p_j/\delta'$ for all $j$. That this defines
an isomorphism follows from the defining relations in Theorem
\ref{thm:present}.
\end{proof}

\begin{rmk}\label{rmk:reps-of-PB}
The representation theory of $\PB_r(\delta,\delta')$ is worked out in
\cite{MM}, using a Morita equivalence result (cf.~\cite{HdM} for
another approach). In particular, it is shown that
$\PB_r(\delta,\delta')$ is cellular \cite{GL} and generically
semisimple; the cell modules are indexed by the set
\[
\Lambda_r = \{\lambda \vdash k: 0 \le k \le r\} 
\]
of partitions of the integers $0, 1, \dots, r$.
% THE FOLLOWING IS FALSE:
% Another possible approach to that representation theory is to notice
% that $\PB_r(\delta,\delta')$ is an iterated inflation of Brauer
% algebras, in the sense of \cites{KX:99,KX:01,Green-Paget}.
\end{rmk}

\section{Schur--Weyl duality}\label{sec:SWD}\noindent
Now we can formulate and prove our main results. We will need the
following general fact, which is presumably known. We include a proof
since we were unable to find a suitable reference.

\begin{lem}\label{lem:closure}
Let $G$ be a subgroup of $\GL(V)$.  Then for any $r \ge 1$, we have
the equality $\End_G(V^{\otimes r}) = \End_{\ov{G}}(V^{\otimes r})$,
where $\ov{G}$ is the Zariski-closure of $G$ in $\GL(V)$.
\end{lem}

\begin{proof}
It suffices to show that $\End_G(V^{\otimes r})
\subset \End_{\ov{G}}(V^{\otimes r})$, the reverse inclusion being
obvious. Fix any $A \in \End_G(V^{\otimes r})$, and define a map
$F: \End(V) \to \End(V^{\otimes r})$ by
\[
F(g) := A(g \otimes \cdots \otimes g) - (g \otimes \cdots \otimes g)A.
\]
Fix a basis $\{v_l\}_{l=1,\dots, n}$ of $V$, and identify $\End(V)
\cong \M_n(\C)$ via the basis; similarly identify $\End(V^{\otimes r})
\cong \M_{n^r}(\C)$ via the corresponding basis
\[
\{v_{j_1} \otimes v_{j_2} \otimes \cdots \otimes v_{j_r}
\}_{j_1,\dots, j_r =1,\dots, n}
\]
of $V^{\otimes r}$. Then $F$ vanishes on $G$, so its matrix coordinate
functions $F_{i,j}$ also vanish on $G$, for any multi-index pair $i =
(i_1,i_2, \dots, i_r)$, $j = (j_1, j_2, \dots, j_r)$. But $F_{i,j}$ is
a complex-valued polynomial function on $\End(V) = \M_n(\C)$ which
vanishes on $G$, hence it vanishes on $\ov{G}$.  So $F$ itself
vanishes on $\ov{G}$. This implies that
\[
A(g \otimes \cdots \otimes g) - (g \otimes \cdots \otimes g)A = 0
\]
for any $g \in \ov{G}$. In other words, $A
\in \End_{\ov{G}}(V^{\otimes r})$, as required. 
\end{proof}

In the following, we always assume that $[n]_q \ne 0$, so that the
orthogonal decomposition $\bE = \bF \oplus \bL$ holds. Hence, by Lemma
\ref{lem:H-structure}, Proposition \ref{prop:splitting}, and
\cite{Chevalley}*{p.~88} the module $\bE^{\otimes r}$ is semisimple as
a $\C[TW_n]$-module. Recall that $\bil{\bl_0}{\bl_0} = [n]_q$ and
define
\[
\uu_0 = [n]_q^{-1/2} \bl_0 .
\]
The assumption $[n]_q \ne 0$ implies that the restriction of the
bilinear form to $\bF$ is nondegenerate (i.e., $\bF$ is a
non-isotropic subspace of $\bE$). Fix any orthonormal basis $\uu_1,
\dots, \uu_{n-1}$ of $\bF$. Then
\[
\{\uu_0\} \cup \{\uu_1, \dots, \uu_{n-1}\}
\]
is an orthonormal basis of $\bE$ which is compatible with the
decomposition $\bE = \bL \oplus \bF$ in the sense that $\bL = \C
\uu_0$ and $\bF= \bigoplus_{j=1}^{n-1} \C \uu_j$.

Define linear endomorphisms $\ov{s}_i$, $\ov{e}_i$ (for $i = 1, \dots,
r-1$) and $\ov{p}_j$ (for $j = 1, \dots, r$) of $\bE^{\otimes r}$ on
basis elements $u=\uu_{k_1} \otimes \cdots \otimes \uu_{k_r}$ (where
$k_\alpha \in \{0,\dots, n-1\}$ for all $\alpha$) as follows:
\begin{enumerate}\renewcommand{\labelenumi}{(\roman{enumi})}
\item $\ov{s}_i(u)$ is the same tensor but with the factors in tensor
positions $i,i+1$ interchanged.

\item $\ov{e}_i(u)$ is $\delta_{k_i,k_{i+1}}$ times the vector obtained by
replacing the factors of $u$ in tensor positions $i,i+1$ by
$\uu_0 \otimes \uu_0 + \cdots + \uu_{n-1}\otimes \uu_{n-1}$.

\item $\ov{p}_j(u)$ is the same tensor but with the factor in tensor
position $j$ replaced by $\pi(\uu_j)$, where $\pi$ is the orthogonal
projection $\bE \twoheadrightarrow \bL$ in Lemma \ref{lem:projection}.
\end{enumerate}
The endomorphism $\ov{e}_i$ defined in (ii) is independent of the
chosen orthonormal basis $\{\uu_1, \dots, \uu_{n-1}\}$. Indeed,
$\ov{e}_i$ corresponds with the identity map $\text{id}_\bE
\in \End(\bE)$ under the canonical isomorphism
\[
\End(\bE) \cong \bE^* \otimes \bE \cong \bE \otimes \bE
\]
coming from the identification $\bE^* \cong \bE$ arising from the
given bilinear form. It is also possible (see \cite{HdM}) to define
the action diagrammatically by a uniform description of the action of
any partial Brauer diagram.

\begin{lem}
The diagonal action of the twin group $TW_n$ on $\bE^{\otimes r}$
commutes with the operators $\ov{s}_i$, $\ov{e}_i$, and $\ov{p}_j$ for
$i = 1, \dots, r-1$, $j = 1, \dots, r$.
\end{lem}

\begin{proof}
That the diagonal $TW_n$-action commutes with $\ov{s}_i$ is obvious,
and commutativity with $\ov{e}_i$ is clear because the $S_i$ act as
orthogonal matrices in $\OO(\bE)$, so we can appeal to the classical
fact \cite{Brauer} that the diagonal $\OO(\bE)$-action commutes with
$\ov{e}_i$. For the commutativity with $\ov{p}_j$, see Lemma
\ref{lem:projection}. 
\end{proof}

\begin{prop}[\cite{MM}*{Prop.~5.1}]\label{prop:MM}
Suppose that $[n]_q \ne 0$. Regard $\bE = \bL \oplus \bF$ as a
representation of $\OO(\bF)$ by restriction from $\OO(\bE)$ to $\OO(\bF)$.
Then the centralizer algebra $\End_{\OO(\bF)}(\bE^{\otimes r})$ for the
diagonal action of $\OO(\bF) \cong \OO_{n-1}(\C)$ on $\bE^{\otimes r}$ is
the algebra $Z(r)$ generated by $\ov{s}_i$, $\ov{e}_i$, and $\ov{p}_j$
for $i = 1, \dots, r-1$, $j = 1, \dots, r$.
\end{prop}

\begin{proof}
The action of $\OO(\bE)$ on $\bL$ is trivial; hence the same is true of
its restriction to $\OO(\bF)$. So $\bL \cong \C$ as
$\C[\OO(\bF)]$-modules. This precisely matches the situation addressed
in \cite{MM}*{Prop.~5.1}.
\end{proof}

\begin{prop}\label{lem:extend}
Suppose that $[n]_q \ne 0$. Let $\OO(\bE) \cong \OO_{n}(\C)$ act naturally
on $\bE$ and diagonally on $\bE^{\otimes r}$, with $\OO(\bL) \times
\OO(\bF)$ acting by restriction. Then:
\begin{enumerate}
\item $\End_{\OO(\bL) \times \OO(\bF)}(\bE^{\otimes r}) = Z(r)$.
\item For any $\delta' \ne 0$, $Z(r)$ is generated by $\ov{s}_i$,
  $\ov{e}_i$, and $\delta'\,\ov{p}_j$ for $i = 1, \dots, r-1$, $j = 1,
  \dots, r$.
\end{enumerate}
\end{prop}

\begin{proof}
(a) Since $\OO(\bL) \cong \{\pm 1\}$ is the cyclic group of order 2, the
  action of $\OO(\bL) \times \OO(\bF)$ differs from the action of $\OO(\bF)$
  only by signs, so
\[
\End_{\OO(\bL) \times \OO(\bF)}(\bE^{\otimes r}) = \End_{\OO(\bF)}(\bE^{\otimes r}).
\]
The result then follows from Proposition \ref{prop:MM}.

(b) Scaling any generator by a nonzero scalar does not affect the
centralizer algebra $Z(r)$.
\end{proof}

\begin{prop}\label{prop:extend}
Assume that $[n]_q \ne 0$, $[n-2]_q^! \ne 0$, and that
\[
  q \ne \frac{-\lambda \pm \sqrt{-1-2\lambda}}{1+\lambda}
\]
for any $\lambda = \cos(2k\pi/m)$ with $m \in \Z_{\ge 0}$.
Then $\End_{TW_n}(\bE^{\otimes r}) = Z(r)$.
\end{prop}

\begin{proof}
Let $G$ be the image of the representation $\rho: TW_n \to
\GL(\bE)$. By Theorem \ref{thm:density} and Lemma
\ref{lem:closure} we have
\[
\End_{TW_n}(\bE^{\otimes r}) = \End_{G}(\bE^{\otimes r}) =
\End_{\ov{G}}(\bE^{\otimes r})
\]
and the result follows from Lemma \ref{lem:extend}(a).
\end{proof}

\begin{thm}\label{thm:SWD}
Assume that $[n]_q \ne 0$, $[n-2]_q^! \ne 0$, and that
\[
  q \ne \frac{-\lambda \pm \sqrt{-1-2\lambda}}{1+\lambda}
\]
for any $\lambda = \cos(2k\pi/m)$ with $m \in \Z_{\ge 0}$. Let $0 \ne
\delta'$ be a complex number.  Regarded as a $(\C[TW_n],
\PB_r(n,\delta'))$-bimodule, $\bE^{\otimes r}$ satisfies Schur--Weyl
duality, in the sense that the enveloping algebra of each action is
equal to the full centralizer of the other. Here $TW_n$ acts
diagonally and the generators $s_i$, $e_i$, $p_j$ of
$\PB_r(n,\delta')$ act as $\ov{s}_i$, $\ov{e}_i$, $\delta'\ov{p}_j$
respectively. Finally, the action of $\PB_r(n,\delta'))$ is faithful
if and only if $n>r$.
\end{thm}

\begin{proof}
We just have to put the pieces together. Proposition
\ref{prop:extend} computes the centralizer algebra $Z(r)
= \End_{TW_n}(\bE^{\otimes r})$ and by verifying that the defining
relations in Theorem \ref{thm:present} are satisfied by the generators
of $Z(r)$, part (c) shows that it is a homomorphic image of the
algebra $\PB_r(n,\delta')$.  This proves one half of Schur--Weyl
duality. The other half follows by standard arguments in the theory of
semisimple algebras, since the enveloping algebra of the orthogonal
group (or its restriction to $TW_n$) action on the semisimple module
$\bE^{\otimes r}$ is a semisimple algebra. The final claim follows from
the dimension comparison in \cite{MM}*{Thm.~5.3(iii)}.
\end{proof}

\begin{rmk}
There are a number of interesting choices of the scaling parameter
$\delta'$ in Theorem \ref{thm:SWD}:

(i) Taking $\delta'=1$ recovers an analogue of the Schur--Weyl duality
result of \cite{MM} in which the action of the orthogonal group
$\OO_{n-1}(\C)$ has been replaced by the action of the twin group
$TW_n$. This connects the representations of $\PB_r(n,1)$ to those of
$TW_n$.
  
(ii) If we take $\delta' = n$ in the theorem, the corresponding
algebra $\PB_r(n,n)$ is contained in the partition algebra $\Ptn_r(n)$
at parameter $n$.

(iii) Choosing $\delta' = [n]_q$ clears denominators in the
corresponding pseudo-projections $[n]_q\ov{p}_j$. By Lemma
\ref{lem:projection}, the entries of the matrix of $[n]_q\ov{p}_j$
with respect to the $\{\ee'_i\}$-basis depend only on nonnegative
integral powers of $\sqrt{q}$.
\end{rmk}

Recall that the irreducible polynomial representations of $\OO_n(\C)$
are typically indexed by the set of all partitions $\lambda$ with not
more than $n$ boxes in the first two columns of the corresponding
Young diagram. Replacing $n$ by $n-1$ we can index the irreducible
polynomial $\OO_{n-1}(\C)$-modules by the set of partitions $\lambda$
satisfying the condition $\lambda'_1+\lambda'_2 \le n-1$, where
$\lambda'$ is the conjugate partition. This gives the following
immediate consequence of Theorem \ref{thm:SWD}. The set $\Lambda_r$ is
defined in Remark \ref{rmk:reps-of-PB}.

\begin{cor}
Under the same hypotheses as in Theorem \ref{thm:SWD}, we have a
decomposition
\[
\bE^{\otimes r} \cong \bigoplus_{\lambda \in \Lambda_r:
  \lambda'_1+\lambda'_2 \le n-1} T^\lambda \otimes B^\lambda
\]
as $(\C[TW_n], \PB_r(n,\delta'))$-bimodules, where $T^\lambda$,
$B^\lambda$ are irreducible representations of $TW_n$,
$\PB_r(n,\delta')$ respectively.
\end{cor}

We obtain a second new instance of Schur--Weyl duality involving the
twin group (for tensor powers of $\bF$) as an immediate consequence
of Theorem \ref{thm:density} and the classical result of
\cite{Brauer}.

\begin{thm}\label{thm:SWD-F}
  Let $\uu_1, \dots, \uu_{n-1}$ be any orthonormal basis of $\bF$ with
  respect to the restriction of the bilinear form $\bil{-}{-}$. Assume
  the hypotheses of Theorem \ref{thm:SWD}. Regarded as a $(\C[TW_n],
  \B_r(n-1))$-bimodule, $\bF^{\otimes r}$ satisfies Schur--Weyl
  duality, with the generator $e_i$ of $\B_r(n-1)$ acting by
  \[
  \uu_{j_1} \otimes \cdots \uu_{j_r} \mapsto \delta_{j_i, j_{i+1}}
  \sum_{k=1}^{n-1} \uu_{j_1} \otimes \cdots \uu_{j_{i-1}} \otimes
  \uu_k \otimes \uu_k \otimes  \uu_{j_{i+2}} \otimes \cdots \uu_{j_r}.
  \]
  and the $s_i$ acting by swapping places $i,i+1$ as usual, for each
  $i = 1, \dots, r-1$. The action of $\B_r(n-1)$ is faithful if and
  only if $n-1 \ge 2r$. 
\end{thm}

\begin{proof}
By \cite{Brauer}, the actions of $\OO_{n-1}(\C)$ and $\B_r(n-1)$ on
$\bF^{\otimes r}$ commute, where $\bF$ is regarded as the (natural)
vector representation of $\OO_{n-1}(\C)$.  Let $G$ be the image of
the representation $TW_n \to \OO(\bF)$. By Lemma \ref{lem:closure} we
have
\[
\End_{TW_n}(\bF^{\otimes r}) = \End_{\ov{G}}(\bF^{\otimes r}) =
\End_{\OO_{n-1}(\C)}(\bF^{\otimes r}).
\]
The result now follows by \cite{Brauer}. 
\end{proof}

%%%%%%%%%%%%%%%%%%%%%%%%%%%%%%%%%%%%%%%%%%%%%%%%%%%%%%%%%%%%%%%%%%%%%
\appendix
\section{An explicit orthonormal basis of $\bF$}\noindent
The proof of Schur--Weyl duality in Section \ref{sec:SWD} requires an
orthonormal basis of $\bF$, which exists by general principles (e.g.,
\cite{Vinberg}*{Thm.~5.46}).  It is sometimes useful to have an
explicit orthonormal basis, so we now construct one using the
Gram--Schmidt orthogonalization procedure applied to the basis
$\{\ff_1, \dots, \ff_{n-1}\}$.  As an application, in Proposition
\ref{prop:alt-density} we obtain a simpler proof of the density
conclusion in Theorem \ref{thm:density}, under somewhat stronger
hypotheses.

To begin the orthogonalization procedure, it is useful to observe that
\begin{equation}\label{eq:fifj}
\bil{\ff_i}{\ff_j} = \delta_{ij}[2]_q - (\delta_{i,j+1} +
\delta_{i,j-1})\sqrt{q} .
\end{equation}
In particular, $\bil{\ff_i}{\ff_j} = 0$ unless $i=j$ or $i,j$ are
adjacent integers. Thus, the matrix of $\bil{-}{-}$ with respect to the
basis $\{\ff_i\}$ is a banded tridiagonal $(n-1) \times (n-1)$ matrix
of the form
\[
A_{n-1} = 
\begin{bmatrix}
  a & b \\
  c & a & b \\
  & \ddots & \ddots & \ddots\\
  &  & c & a & b \\
  & & & c & a
\end{bmatrix}
\]
where $a= [2]_q$, $b=c= -\sqrt{q}$. An easy inductive argument shows that
\begin{equation}
  \det A_{n-1} = [n]_q .
\end{equation}
For any $k\le n-1$, let $\bF_k = \C\ff_1 \oplus \cdots \oplus \C\ff_k$
and let $A_k$ be the matrix of the restriction of $\bil{-}{-}$ to
$\bF_k$. The matrix $A_k$ is the upper left $k \times k$ submatrix of
$A_{n-1}$. Put $\bF_0 = 0$, $d_0 = 1$, and set
\[
d_k = \det A_k = [k+1]_q \quad \text{for all } k\ge 1.
\]
We obtain the following result from the standard Gram--Schmidt
orthogonalization procedure (see e.g.~\cite{Vinberg}*{Thm.~5.47}).

\begin{lem}\label{lem:Gram-Schmidt}
  Suppose that $d_k=[k+1]_q \ne 0$ for $k=1, \dots, n-1$ (that is,
  $[n]^!_q \ne 0$). Then there exists a unique orthogonal basis
  $\vv_1, \dots, \vv_{n-1}$ of $\bF$ such that
  \[
  \vv_k \in \ff_k + \bF_{k-1}, \quad \text{for all } k = 1, \dots, n-1.
  \]
  Furthermore, $\bil{\vv_k}{\vv_k} = [k+1]_q/[k]_q$ for all $k = 1,
  \dots, n-1$.
\end{lem}

Under the hypotheses of Lemma \ref{lem:Gram-Schmidt}, the fact that
$\vv_j \in \ff_j + \bF_{j-1}$ in light of \eqref{eq:fifj} immediately
implies that
\begin{equation}\label{eq:f_iv_j}
  \bil{\ff_i}{\vv_{i-1}} = -\sqrt{q} \quad\text{and}\quad
  \bil{\ff_i}{\vv_j} = 0 \text{ for all } j<i-1.
\end{equation}
Thus the Gram--Schmidt formula yields the relation
\begin{equation}\label{eq:vi}
  \vv_i = \ff_i + \frac{\sqrt{q}}{\bil{\vv_{i-1}}{\vv_{i-1}}} \vv_{i-1}
  =  \ff_i + \frac{q^{1/2}[i-1]_q}{[i]_q} \vv_{i-1}.
\end{equation}
This can be used recursively (with $\vv_1=\ff_1$) to compute the
transition coefficients expressing $\vv_i$ as a linear combination of
the $\ff_j$, but it is slightly simpler to rewrite \eqref{eq:vi} in
the equivalent form
\begin{equation}\label{eq:vi-equiv}
  [i]_q\vv_i = [i]_q\ff_i + q^{1/2}[i-1]_q \vv_{i-1}.
\end{equation}
We summarize our conclusions.

\begin{lem}\label{lem:v-prime}
Define $\vv'_i = [i]_q\vv_i$ and suppose that $[n]^!_q \ne 0$.  Then
$\vv'_1, \dots, \vv'_{n-1}$ is another orthogonal basis of $\bF$
satisfying:
\begin{enumerate}
\item $\vv'_1 = \vv_1 = \ff_1$.
\item $\vv'_i = [i]_q\ff_i + q^{1/2}\vv'_{i-1}$ for all $i=2, \dots, n-1$. 
\end{enumerate}
\end{lem}

Note that $\bil{\vv'_i}{\vv'_i} = [i]_q [i+1]_q$ for all $i$.  Solving
the recurrence relation in Lemma \ref{lem:v-prime} gives the 
explicit formulas
\begin{equation}
\vv'_i = \sum_{j=1}^i q^{(i-j)/2} [j]_q \ff_j = -[i]_q\ee'_{i+1} +
\sum_{j=1}^i q^{(i+j-1)/2} \ee'_j
\end{equation}
expressing the $\vv'_i$ in terms of either the $\{\ff_j\}$ or the
$\{\ee'_j\}$.

\begin{lem}\label{lem:fivj-prime}
  Suppose that $[n]^!_q \ne 0$. Then
  \[
  \bil{\ff_i}{\vv'_{i-1}} = -q^{1/2}[i-1]_q, \quad
  \bil{\ff_i}{\vv'_i} = [i+1]_q
  \]
  and $\bil{\ff_i}{\vv'_j} = 0$ for all $j \ne i, i-1$. 
\end{lem}

\begin{proof}
The value of $\bil{\ff_i}{\vv_j}$ for all $j = 1, \dots, i-1$ was
computed in \eqref{eq:f_iv_j}, which yields the result in those cases.
Lemma \ref{lem:v-prime}(b) implies that when $j = i$,
\[
\bil{\ff_i}{\vv'_i} = [i]_q \bil{\ff_i}{\ff_i} + q^{1/2}
\bil{\ff_i}{\vv'_{i-1}} = [i]_q[2]_q - q[i-1]_q = [i+1]_q .
\]
and similarly when $j = i+1$,
\begin{align*}
\bil{\ff_i}{\vv'_{i+1}} &= [i+1]_q \bil{\ff_i}{\ff_{i+1}} + q^{1/2}
\bil{\ff_i}{\vv'_{i}} \\ &= -q^{1/2}[i+1]_q +  q^{1/2} [i+1]_q = 0.
\end{align*}
By repeating the argument, the above equality inductively implies
that $\bil{\ff_i}{\vv'_j} = 0$ for any $j>i+1$. 
\end{proof}

By scaling the orthogonal basis $\vv'_1, \dots, \vv'_{n-1}$ we obtain
the desired orthonormal basis $\uu_1, \dots, \uu_{n-1}$ of $\bF$,
where
\begin{equation}
  \uu_i = [i]_q^{-1/2} [i+1]_q^{-1/2} \vv'_i
\end{equation}
Now we compute the matrices expressing the action of the operators
$S_i$, defined in \eqref{eq:twin-rels}, with respect to the
orthonormal basis $\{\uu_j\}$.

\begin{lem}\label{lem:action-on-u}
Suppose that $[n]^!_q \ne 0$. Then $S_i$ fixes $\uu_j$ for all $j \ne
i-1,i$ and
\begin{align*}
  S_i \cdot \uu_{i-1} & = a_i \uu_{i-1} + b_i \uu_i \\
  S_i \cdot \uu_i & = b_i \uu_{i-1} - a_i \uu_i 
\end{align*}
where
\[
a_i = \frac{[2]_{q^i}}{ [2]_q [i]_q }, \quad b_i = \frac{2\sqrt{q}\,
  [i+1]_q^{1/2} \,[i-1]_q^{1/2}}{ [2]_q [i]_q }.
\]
Here $a_i^2+b_i^2 = 1$. Moreover, $a_i = [i]_{q^2} [i]_q^{-2}$ is an
alternative expression for $a_i$.
\end{lem}

\begin{proof}
Observe that we know all the $\bil{\ff_i}{\uu_j}$ from Lemma
\ref{lem:fivj-prime} and the definition of $\uu_j$; in particular,
$\bil{\ff_i}{\uu_j} = 0$ unless $j$ is equal to $i-1$ or $i$.
By Lemma \ref{lem:action-on-f-basis} we have
\[
S_i \cdot \uu_j = \uu_j - 2\proj{\ff_i}{\uu_j}. 
\]
Hence $S_i \cdot \uu_j = \uu_j$ if $j \ne i,i-1$. This proves the
first claim. It only remains to calculate the $a_i$, $b_i$. For
example,
\begin{align*}
a_i &= \bil{S_i \cdot \uu_{i-1}}{\uu_{i-1}} =
\bil{\uu_{i-1}}{\uu_{i-1}} - 2
\frac{\bil{\ff_i}{\uu_{i-1}}}{\bil{\ff_i}{\ff_i}}
\bil{\ff_i}{\uu_{i-1}} \\
&= 1 - 2 \frac{[i-1]_q^{-1} [i]_q^{-1} \bil{\ff_i}{\vv'_{i-1}}^2}{[2]_q}
= 1 - \frac{2q [i-1]_q}{[2]_q[i]_q} = \frac{[2]_{q^i}}{[2]_q[i]_q}.
\end{align*}
The calculation of $b_i$ is similar. The proof of the alternative
formula for $a_i$ is an easy exercise.
\end{proof}

The matrix of the action of $S_i$ on the orthonormal $\{\uu_j\}$-basis
is of the block diagonal form $\Delta_i = \text{diag}(I_{i-1}, \Delta'_i,
I_{n-i-2})$, where
\[
\Delta'_1 = 
\begin{bmatrix}
-1 & 0 \\ 0 & 1
\end{bmatrix}
\text{ and }
\Delta'_i = 
\begin{bmatrix}
a_i & b_i \\ b_i & -a_i
\end{bmatrix} \text{ if } i>1 .
\]
Hence $\Delta_1\Delta_2 = \text{diag}(I_{i-1}, \Delta'_1 \Delta'_2,
I_{n-i-2})$ is a rotation matrix, where
\[
\Delta'_1 \Delta'_2 =
\begin{bmatrix}
-a_2 & -b_2 \\ b_2 & -a_2
\end{bmatrix} .
\]
We will also need the diagonal matrices
\[
D_i = \text{diag}(1, \dots, 1, -1, 1, \dots, 1)
\]
with the $-1$ appearing in the $i$th diagonal entry.
Note that $D_1 = \Delta_1$ is in $G$. 

We have the following variant of Theorem \ref{thm:density} (under
slightly stronger hypotheses).

\begin{prop}\label{prop:alt-density}
Suppose that $[n]^!_q \ne 0$. Let $G$ be the image of the
representation $TW_n \to \OO(\bF)$, and $\ov{G}$ its Zariski-closure in
$\OO(\bF)$.  Suppose that the matrix $D_i \Delta_{i+1}$ has infinite
order for each $i = 1, \dots, n-2$. Then for all $k = 2,\dots, n-1$:
  \begin{enumerate}
  \item $\ov{G}$ contains $\mathrm{diag}(I_{k-2}, \OO_2(\C), I_{n-1-k})$.
  \item $D_k$ belongs to $\ov{G}$.
  \end{enumerate}
Thus $\ov{G} = \OO(\bF)$.
\end{prop}

\begin{proof}
As noted above, $D_1 = \Delta_1$ is in $G$, so $D_1 \Delta_2$ belongs
to $G$. The group generated by $D_1 \Delta_2$ is an infinite cyclic
subgroup of the one-parameter group $\text{diag}(\SO_2(\C), I_{n-3})$,
so its Zariski-closure is equal to $\text{diag}(\SO_2(\C), I_{n-3})$,
and this is contained in $\ov{G}$. It follows that
$\text{diag}(\OO_2(\C),I_{n-3}) \subset \ov{G}$, proving (a) for
$k=2$. This implies (b) for the case $k=2$.

For $k>2$ we may assume by induction that (a), (b) hold for $k-1$.  In
particular, $D_{k-1} \in \ov{G}$. Repeat the argument to see that the
Zariski-closure of the group generated by $D_{k-1} \Delta_k$ is
$\text{diag}(I_{k-2}, \SO_2(\C), I_{n-1-k})$, and this is contained in
$\ov{G}$. It follows that $\text{diag}(I_{k-2}, \OO_2(\C),I_{n-1-k}) \subset
\ov{G}$, which proves (a) and (b) by induction.

Let $G(k)$ be the group generated by $\Delta_1, \dots,
\Delta_k$. Notice that $G = G(n-1)$. Assume by induction that
$\ov{G(k-1)} = \text{diag}(\OO_{k-1}(\C), I_{n-k})$. Thus
\begin{equation}\label{eq:lie-gens}
e_{i,j} - e_{j,i} \in \Lie \ov{G(k-1)} \text{ for all } 1\le i<j \le
k-1
\end{equation}
and $\Lie \ov{G(k-1)} \cong \so_{k-1}(\C)$ is contained in $\Lie
\ov{G(k)}$. By (a) we know that $\Lie \ov{G(k)}$ contains $X =
e_{k-1,k}-e_{k,k-1}$. By taking commutators of $X$ with the elements
in \eqref{eq:lie-gens} we see that $\Lie \ov{G(k)} \cong
\so_k(\C)$. By induction, this holds for all $k$. 

Hence, $\Lie \ov{G} = \Lie \ov{G(n-1)} = \so_{n-1}(\C)$. As $\ov{G}
\subset \OO_{n-1}(\C)$ and contains reflections (elements of determinant
$-1$; e.g., any $D_k$) it follows that $\ov{G} = \OO_{n-1}(\C) \cong
\OO(\bF)$.
\end{proof}

\begin{rmk}
The values of $q$ making $D_i\Delta_{i+1}$ of finite order can be
analyzed by means of Chebyshev polynomials, similar to 
calculations in Section \ref{sec:density}. We omit the details.
\end{rmk}

%%%%%%%%%%%%%%%%%%%%%%%%%%%%%%%%%%%%%%%%%%%%%%%%%%%%%%%%%%%%%%%%%%%%%%%
% bibliography using amsrefs package 
%%%%%%%%%%%%%%%%%%%%%%%%%%%%%%%%%%%%%%%%%%%%%%%%%%%%%%%%%%%%%%%%%%%%%%%

\end{document}